\documentclass[fleqn,12pt]{article}
\usepackage{amsmath,amssymb,amsthm,esint,dsfont,xcolor}

\usepackage[margin=3cm]{geometry}
\usepackage[english]{babel}
\usepackage{csquotes}
\usepackage[hidelinks]{hyperref}
\newcommand{\norm}[1]{{\left\Vert #1\right\Vert}}
\newcommand{\abs}[1]{{\left\vert #1\right\vert}}
\newcommand{\e}{\varepsilon}
\newcommand{\R}{{\mathbb R}}
\newcommand{\one}{{\mathds{1}}}

\DeclareMathOperator{\dist}{dist}
\DeclareMathOperator{\diam}{diam}

\DeclareMathOperator{\Lip}{Lip}
\usepackage[hidelinks]{hyperref}

\newtheorem*{prop*}{Proposition}

\title{On the regularity of weak solutions to Burgers' equation with finite entropy production}
\date{}
\author{Xavier Lamy 
\thanks{Institut de Math\'ematiques de Toulouse, Universit\'e Paul Sabatier, Toulouse, France. Part of this work was conducted while XL was a postdoctoral researcher at the Max Planck Institute for Mathematics in the Sciences, Leipzig, Germany. Email: xlamy@math.univ-toulouse.fr}
 \and Felix Otto
\thanks{Max Planck Institute for Mathematics in the Sciences, Leipzig, Germany. Email: otto@mis.mpg.de} 
 }

\newtheorem{thm}{Theorem}
\newtheorem{prop}[thm]{Proposition}
\newtheorem{lem}[thm]{Lemma}
\newtheorem{cor}[thm]{Corollary}
\theoremstyle{definition}
\newtheorem{rem}[thm]{Remark}

\begin{document}
\maketitle

\abstract{Bounded weak solutions of Burgers' equation $\partial_tu+\partial_x(u^2/2)=0$ that are not entropy solutions need in general not be $BV$. Nevertheless it is known that solutions with finite entropy productions have a $BV$-like structure: a rectifiable jump set of dimension one can be identified, outside which $u$ has vanishing mean oscillation at all points. But it is not known whether all points outside this jump set are Lebesgue points, as they would be for $BV$ solutions. In the present article we show that the set of non-Lebesgue points of $u$ has Hausdorff dimension at most one. In contrast with the aforementioned structure result, we need only one particular entropy production to be a finite Radon measure, namely $\mu=\partial_t (u^2/2)+\partial_x(u^3/3)$. We prove H\"older regularity at points where $\mu$ has finite $(1+\alpha)$-dimensional upper density for some $\alpha>0$. The proof is inspired by a result of De Lellis, Westdickenberg and the second author : if $\mu_+$ has vanishing 1-dimensional upper density, then $u$ is an entropy solution. We obtain a quantitative version of this statement: if $\mu_+$ is small then $u$ is close in $L^1$ to an entropy solution.}

\section{Introduction}\label{s:intro}

It is well-known that weak solutions of Burgers' equation
\begin{equation}\label{eq:burgers}
\partial_t u+\partial_x \frac{u^2}{2}=0,
\end{equation}
(and more generally scalar conservation laws)
are not uniquely determined by initial data, and this is the reason why the notion of entropy solution was introduced \cite{kruzkov70}.
Entropy solutions are characterized by their nonpositive entropy production : for any convex entropy $\eta\colon\R\to\R$ and associated entropy flux $q(u)=\int^u v\eta'(v)dv$, the corresponding entropy production $\mu_\eta$ satisfies
\begin{equation*}
\mu_\eta = \partial_t \eta(u) +\partial_x q(u) \leq 0.
\end{equation*}
This constraint  ensures well-posedness of the Cauchy problem for \eqref{eq:burgers} with $L^\infty$ initial data.
Entropy solutions can be equivalently characterized by Oleinik's estimate
$\partial_x u \leq 1/t$ \cite{oleinik57},
and in particular they are locally in $BV$.

Although entropy solutions are the physically relevant solutions, general weak solutions sometimes need to be considered. 
{
For instance in \cite{varadhan04,mariani10,bellettinietal10} large deviation principles for stochastic approximation of entropy solutions are
related to variational principles for energy functionals of the form
\begin{equation*}
F_\e(u)=\int \abs{\frac{1}{\e}\abs{\partial_x}^{-1}\left(\partial_t u +\partial_x \frac{u^2}{2}\right) -\e \partial_x u}^2.
\end{equation*}
The $\Gamma$-limit of such functional is defined for weak solutions of \eqref{eq:burgers} that need not be entropy solutions, but have finite entropy production:}
\begin{equation}\label{eq:finiteentropy}
\mu_\eta = \partial_t \eta(u) +\partial_x q(u)\quad\text{ is a locally finite Radon measure,}
\end{equation}
for any $\eta\in C^2(\R)$ and associated flux $q$.
An important feature of such solutions is that they enjoy a kinetic formulation (see e.g. \cite{DOW03}), namely there exists $m(t,x,v)$ a locally finite Radon measure such that
\begin{equation*}
\partial_t\chi +v\partial_x\chi = \partial_v m,\quad\chi(t,x,v)=\one_{0<v\leq u(t,x)} - \one_{u(t,x)\leq v < 0}.
\end{equation*}
The measure $m$ encodes the entropy production through the formula
\begin{equation*}
\langle \mu_\eta,\varphi\rangle  = \int \eta''(v) \varphi(t,x) m(dt,dx,dv).
\end{equation*}
For entropy solutions it is nonpositive and the kinetic formulation was introduced in \cite{lions94}.

Another motivation for studying general weak solutions of \eqref{eq:burgers} comes from a formal analogy with solutions of the eikonal equation
\begin{equation}\label{eq:eikonal}
\abs{\nabla \varphi}=1,
\end{equation}
that need not be viscosity solutions. Such solutions arise for instance in the problem of $\Gamma$-convergence of the Aviles-Giga functional
\begin{equation*}
E_\e(\varphi)=\frac{\e}{2}\int\abs{\nabla^2\varphi}^2 +\frac{1}{2\e}\int\left(\abs{\nabla\varphi}^2-1\right)^2.
\end{equation*}
They can be endowed with a relevant concept of entropy production \cite{jinkohn00,ADM99,DMKO01,ignatmerlet12} and a kinetic formulation \cite{jabinperthame01,JOP02}. The $\Gamma$-limit of $E_\e$ is conjectured to be the total entropy production, but a proof of the upper bound is still missing because not enough is known about the regularity of solutions with finite entropy production (see \cite{contidelellis07,poliakovsky10} when $\nabla\varphi\in BV$). 
The analogy between \eqref{eq:burgers} and \eqref{eq:eikonal} has already proven fruitful. For instance, techniques developed in \cite{delellisotto03} to understand the fine structure of solutions of \eqref{eq:eikonal} were adapted in \cite{DOW03} to the context of scalar conservation laws. See also \cite{jabinperthame02,delellisignat15} for other regularity properties shared by both equations.

Unlike entropy solutions, weak solutions of \eqref{eq:burgers} with finite entropy production \eqref{eq:finiteentropy} may not be in $BV$. They are in $B^{1/3}_{3,\infty}$ \cite{golseperthame13}, but this is the best regularity one could hope for \cite{delelliswestdickenberg03}. However it is shown in \cite{lecumberry,DOW03} (related results can be found e.g. in \cite{ambrosioetal02,riviere02}) that they do enjoy a $BV$-like structure, namely: there exists an $\mathcal H^1$-rectifiable set $\mathcal J\subset \Omega$ such that $u$ has strong one-sided traces on $\mathcal J$, and vanishing mean oscillation at all points outside $\mathcal J$. Moreover the entropy production restricted to the \enquote{jump set} $\mathcal J$ can be computed with the chain rule: if $\nu$ denotes a normal vector along $\mathcal J$ and $u^{\pm}$ the corresponding one-sided traces of $u$, then
\begin{equation*}
\mu_\eta \lfloor \mathcal J = \left[ (\eta(u^+)-\eta(u^-))\nu_t + (q(u^+)-q(u^-))\nu_x\right]\mathcal H^1\lfloor \mathcal J.
\end{equation*}
The similarity with the structure of $BV$ solutions is not perfect, and the two following questions are left open: 
\begin{itemize}
\item Is $\mu_\eta$ supported on $\mathcal J$ ?
\item Is every point outside $\mathcal J$ a Lebesgue point of $u$ ?
\end{itemize}
In the present article we investigate the second question. Note that for entropy
solutions of a large class of one-dimensional scalar conservation laws the corresponding questions have been answered positively \cite{delellisriviere03}.

The quadratic entropy $\eta(u)=u^2/2$ plays a special role in our analysis. In fact our methods are strongly inspired by \cite{DOW04} where the importance of that particular entropy is shed light upon. We consider bounded weak solutions $u(t,x)$ of \eqref{eq:burgers} in a domain $\Omega\subset\R_t\times \R_x$ and denote simply by $\mu$ the corresponding entropy production
\begin{equation}\label{eq:mu}
\mu=\partial_t \frac{u^2}{2} +\partial_x \frac{u^3}{3}\;\in\mathcal M(\Omega).
\end{equation}
In \cite{DOW03} the singular set $\mathcal J$ is defined as the set of points with positive upper $\mathcal H^1$ density with respect to the measure $\nu\in\mathcal M(\Omega)$ given by
\begin{equation}\label{eq:nu}
\nu(A)=\abs{m}(A\times\R)=\sup_{\abs{\eta''}\leq 1}\abs{\mu_\eta}(A)\qquad\text{for }A\subset\Omega.
\end{equation}
In other words, denoting by $Q_r(z)$ the square of size $r$ centered at $z$, i.e.
\begin{equation*}
Q_r(z)=(t-r,t+r)\times (x-r,x+r)\quad\text{if }z=(t,x),
\end{equation*}
 the \enquote{regular points} of \cite{DOW03} are those belonging to
\begin{equation*}
\mathcal J^c =\left\lbrace z\in\Omega\colon \lim_{r\to 0} r^{-1}\nu( Q_r(z))=0\right\rbrace\subset \left\lbrace z\in\Omega\colon \lim_{r\to 0} r^{-1}\abs{\mu}( Q_r(z))=0\right\rbrace.
\end{equation*}
The last inclusion follows from $\abs{\mu}\leq\nu$, and it is not clear weather it is strict or not.
\begin{rem}
For $BV$ solutions of \eqref{eq:burgers} the measures $\mu$, $m$ and $\nu$ can be computed explicitly using the chain rule (see e.g. \cite[Remark~2.7]{bellettinietal10}) and one can check that $\nu=\abs{\mu}$ so that the inclusion is not strict.
\end{rem}
In the present paper we need a geometric rate of decay for $r^{-1}\abs{\mu}( Q_r(z))$ -- but no bound on $r^{-1}\nu( Q_r(z))$ -- to conclude that $z$ is a Lebesgue point: our {regular points} are given by
\begin{equation*}
\widetilde{\mathcal J}^c =\left\lbrace z\in\Omega\colon r^{-1}\abs{\mu}( Q_r(z))=\mathcal O(r^{\alpha})\text{ for some }\alpha>0\right\rbrace.
\end{equation*}
In order to quantify the regularity we obtain outside of $\widetilde{\mathcal J}$ we define, for any $\alpha,K>0$, 
\begin{equation}\label{eq:omegaalphaK}
\Omega_{\alpha,K}=\left\lbrace z\in\Omega\colon\abs{\mu}( Q_r(z))\leq K r^{1+\alpha},\;\forall r\in (0,d(z)\wedge 1) \right\rbrace,
\end{equation}
where $d(z)=\dist_\infty(z,\Omega^c)$ denotes the distance of $z$ to the boundary of $\Omega$ with respect to the $\ell^\infty$ norm.
In particular
\begin{equation*}
\widetilde{\mathcal J}=\bigcap_{\alpha,K>0}(\Omega_{\alpha,K})^c
\end{equation*}
has Hausdorff dimension at most one
 since 
{
$\mathcal H^{1+\alpha}((\Omega_{\alpha,K})^c)\lesssim K^{-1}\abs{\mu}(\Omega)$, as follows from a covering argument (see e.g. \cite[Theorem~2.56]{ambrosiofuscopallara}). In} $\Omega_{\alpha,K}$ the function $u$ is H\"older continuous:
\begin{thm}\label{thm:lebquantit}
For any bounded weak solution $u$ of \eqref{eq:burgers} with finite entropy production \eqref{eq:mu}, any $\alpha,K>0$ and $z\in \Omega_{\alpha,K}$ it holds
\begin{equation*}
\fint_{Q_r(z)}\abs{u-\fint_{Q_r(z)} u}\leq C r^{\frac{\alpha}{256}}, \qquad\text{for all } r\in (0,d(z)),
\end{equation*}
where $C>0$ depends on
$\norm{u}_{L^\infty}$, $K$, $\alpha$ and $d(z)$.
\end{thm}

{
\begin{rem}
Such Campanato decay implies local H\"{o}lder continuity in $\Omega_{\alpha,K}$ in the classical sense: for any $d_0>0$ and $z_1,z_2\in \Omega_{\alpha,K}\cap \lbrace d\geq d_0\rbrace$ it holds
\begin{equation*}
\abs{u(z_1)-u(z_2)}\leq C \abs{z_1-z_2}^{\frac{\alpha}{256}},
\end{equation*}
where $C>0$ depends on $\norm{u}_{L^\infty}$, $K$, $\alpha$ and $d_0$. Also, note that the explicit dependence on $K$, $\alpha$ and $d_0$ can be infered from the proof in \S\ref{s:proofs}.
\end{rem}
}

\begin{cor}\label{cor:leb}
The set of non-Lebesgue points of any bounded weak solution $u$ of \eqref{eq:burgers} with finite entropy production \eqref{eq:mu} has Hausdorff dimension at most one.
\end{cor}

The proof of Theorem~\ref{thm:lebquantit} relies on the following principle : if the positive part of the entropy production \eqref{eq:mu} is small, then $u$ should be close to an entropy solution. This principle is already present in \cite{DOW04} where it is shown that {if $\mu_+$ has vanishing upper $\mathcal H^1$-density, then $u$ must be an entropy solution. Here we obtain, using methods  inspired by \cite{DOW04}, a quantitative version of this result: (a small power of) the total mass of $\mu_+$ controls the $L^1$-distance of $u$ to entropy solutions.} This is the content of the next result, where we write $Q_r$ for $Q_r(0,0)$.

\begin{thm}\label{thm:errorentropy}
Let  $u\in L^\infty(Q_1)$ be a weak solution of \eqref{eq:burgers} with finite entropy production \eqref{eq:mu}. Then there exists an entropy solution $\zeta\in L^\infty(Q_1)$  of \eqref{eq:burgers} such that
\begin{equation*}
\int_{Q_{3/4}}\abs{u-\zeta}\leq C\left(1+\abs{\mu}(Q_1)\right)^{\frac 15}\mu_+(Q_1)^{\frac{1}{64}},
\end{equation*} 
where $C>0$ only depends on $\norm{u}_{L^\infty}$.
\end{thm}

To prove Theorem~\ref{thm:errorentropy}, the main step is to estimate the distance to entropy solutions in a rather weak sense, as explained below. This weak estimate can then be strengthened to an $L^1$ estimate by appropriately quantifying the compactness enforced by  \eqref{eq:mu}. 

{
\begin{rem}\label{rem:errorentropynu}
If the measure $\nu$ defined in \eqref{eq:nu}, that encodes all entropy productions,  is finite, then we have a $B^{1/3}_{3,\infty}$ estimate \cite{golseperthame13} so that the compactness is easily quantified. Such an assumption would simplify the proof of Theorem~\ref{thm:errorentropy} and improve the dependence on $\mu_+(Q_1)$ : we would obtain
\begin{equation*}
\int_{Q_{3/4}}\abs{u-\zeta}\leq C \mu_+(Q_1)^{\frac{1}{24}},
\end{equation*}
for some constant $C>0$ depending on $\norm{u}_{L^\infty}$ and $\nu(Q_1)$. Modifying the definition of the sets $\Omega_{\alpha,K}$ accordingly (i.e. incorporating the constraint $r^{-1}\nu(Q_r(z))\leq K$) this would also yield in Theorem~\ref{thm:lebquantit} an exponent $\alpha/96$ instead of $\alpha/256$.
\end{rem}
}

As in \cite{DOW04} we make use of the correspondence between Burgers' equation and the related Hamilton-Jacobi equation
\begin{equation}\label{eq:HJ}
\partial_t h+\frac 12(\partial_x h)^2 =0,
\end{equation}
obtained from observing that the vector field $(-u^2/2,u)$ is curl-free and therefore can be written as a gradient field $(\partial_t h,\partial_x h)$.
The aforementioned weak estimate consists in estimating
 the $L^\infty$-distance of $h$ to viscosity solutions of \eqref{eq:HJ}, which correspond to entropy solutions of \eqref{eq:burgers} \cite{DOW04}.
This is the very heart of our argument and we achieve this in \S\ref{s:HJ} by turning the following loose statement into a rigorous one: if the positive part of the entropy production is small, then $h$ is \enquote{not far} from being a viscosity supersolution of 
\begin{equation*}
\partial_t h +\frac 12(\partial_x h)^2\geq -\delta,\qquad\text{{for \enquote{small} }}\delta.
\end{equation*}
If $h$ really was a viscosity supersolution of such modified \eqref{eq:HJ}, then the comparison principle \cite{crandalllions83} would allow to estimate its $L^\infty$-distance to viscosity solutions. We prove instead a weak version of the maximum principle (Lemma~\ref{lem:supersol}) where we need to assume some additional regularity on the subsolution to compare $h$ with, but this turns out to be sufficient for our purposes.

The plan of the article is as follows. In \S\ref{s:HJ} we derive the estimates for $h$, and in \S\ref{s:proofs} we prove Theorems~\ref{thm:errorentropy} and \ref{thm:lebquantit}.

\section{Estimates for the Hamilton-Jacobi equation}\label{s:HJ}

We denote by $Q$ the unit square
\begin{equation*}
Q:=(0,1)_t\times (0,1)_x,
\end{equation*}
and consider Lipschitz functions $h$ in $\overline Q$ that solve \eqref{eq:HJ} almost everywhere. In particular this ensures \cite[\S{11.1}]{lions} that $h$ restricted to the parabolic boundary
\begin{equation*}
\partial_0 Q:=\lbrace 0\rbrace_t\times (0,1)_x \; \cup \; (0,1)_t\times\lbrace 0,1\rbrace_x,
\end{equation*}
is compatible with the existence of a viscosity solution $\bar h$ satisfying $\bar h=h$ on $\partial_0 Q$. Moreover, such viscosity solution satisfies $\bar h\geq h$ and
\begin{equation}\label{eq:lipbarh}
\abs{\partial_x \bar h}\leq \norm{\partial_x h}_{L^\infty(Q)}.
\end{equation}
{For a proof of \eqref{eq:lipbarh} see Appendix~\ref{a:lipvisc}.}
The main result of this section is the following estimate for $\norm{h-\bar h}_\infty$.

\begin{prop}\label{prop:errorvisc}
Let {$t_1\in [0,1)$} and {$L\geq 1$} be fixed.
There exists a constant $C>0$ such that, for any {function $h$} with $\Lip(h)\leq L$ solving
\begin{equation*}
\partial_t h +\frac 12 (\partial_x h)^2=0\qquad{\text{a.e. in }Q},
\end{equation*}
if $u=\partial_x h$ is such that $\mu=\partial_t (u^2/2)+\partial_x(u^3/3)$ is a Radon measure in $Q$,
then  it holds
\begin{equation}
\sup_{Q\cap\lbrace t\leq t_1\rbrace}\abs{h-\bar h}\leq C {\norm{\mu_+}^{1/8}},
\end{equation}
where $\bar h$ is the viscosity solution of 
\begin{equation*}
\partial_t \bar h+\frac 12(\partial_x \bar h)^2=0,\qquad \bar h=h\text{ on }\partial_0 Q.
\end{equation*}
\end{prop}

As explained in the introduction, the proof of Proposition~\ref{prop:errorvisc} is about showing that if $\mu_+$ is small, then $h$ is \enquote{not far} from being a viscosity supersolution of  \eqref{eq:HJ} with small negative right-hand side. Such property is interesting because super- and subsolutions in the viscosity sense enjoy a comparison principle. In fact instead of proving a supersolution property, we directly prove a comparison principle. The main difference with the comparison principle for viscosity solutions is that we have to assume some additional regularity on the subsolution we are comparing $h$ with, namely semiconvexity. We say that a function $\zeta$ is $(1/r)$-semiconvex if for all points $z,z'$ it holds
\begin{equation*}
\zeta(\theta z +(1-\theta)z')-\theta\zeta(z)-(1-\theta)\zeta(z')\leq\frac 1r,\qquad 0\leq\theta\leq 1.
\end{equation*}
This is equivalent to $\zeta(z)+\abs{z}^2/(2r)$ being convex, and allows to prove the following maximum principle.

\begin{lem}\label{lem:supersol}
For all {$L\geq 0$} there exists a constant $C>0$ such that for any convex open set $U\subset\R_t\times \R_x$ the following holds true. 
\begin{itemize}
\item
Let {a function $h$ with $\Lip(h)\leq L$} solve
\begin{equation*}
\partial_t h +\frac 12 (\partial_x h)^2=0\qquad\text{a.e. in }U,
\end{equation*}
and, denoting $u=\partial_x h$ and $\mu=\partial_t (u^2/2)+\partial_x(u^3/3)$, assume that $\mu$ is a Radon measure in $U$.
\item For some {$\delta,r\in (0,1]$}, let {a function $\zeta$} with $\Lip(\zeta)\leq L$ be a viscosity subsolution of
\begin{equation*}
\partial_t\zeta +\frac 12(\partial_x\zeta)^2\leq -\delta\qquad\text{ in }U,
\end{equation*}
with the additional regularity assumption that $\zeta$ be $(1/r)$-semiconvex.

\item If the positive entropy production is small enough in the sense that
\begin{equation}\label{eq:mu+small}
C \norm{\mu_+} \leq \delta^7 r,
\end{equation}
then for any $(t_0,x_0)\in U$ and $\rho>0$ such that $B_\rho(t_0,x_0)\subset U$ and $\rho\geq \delta^2 r$, the function $(h-\zeta)$ restricted to $B_\rho(t_0,x_0)$ cannot attain its minimum at $(t_0,x_0)$.
\end{itemize}
\end{lem}

With Lemma~\ref{lem:supersol} at hand, Proposition~\ref{prop:errorvisc} will follow by regularizing $\bar h$ (using $\sup$-convolution) and appropriately balancing the scale of regularization with the smallness of $\mu_+$ and the smallness of the negative right-hand side modification of \eqref{eq:HJ}.

\begin{proof}[Proof of Lemma~\ref{lem:supersol}.] We assume that $L=1$, the general case entailing no additional difficulty. Suppose that \eqref{eq:mu+small} holds for some constant $C>0$, and that $(h-\zeta)$ restricted to some ball $B_\rho(t_0,x_0)\subset U$ with $\rho\geq \delta^2r$, attains its minimum at $(t_0,x_0)$. We are going to obtain a contradiction if  the constant $C$ is large enough. We use coordinates in which $(t_0,x_0)=(0,0)$, and assume without loss of generality that $h(0,0)=\zeta(0,0)=0$. 

\textbf{Step 1.} There exists an affine function $\zeta_a$ {with $\Lip(\zeta_a)\leq 1$ and} such that
\begin{gather}\label{eq:affinesubsol}
\partial_t\zeta_a +\frac 12 (\partial_x \zeta_a)^2 \leq -\delta,\\
\label{eq:subgradient}
\zeta_a(0,0)=\zeta(0,0),\qquad \zeta(t,x)\geq \zeta_a(t,x)-\frac{1}{2r}(t^2+x^2).
\end{gather}

At any point $(t,x)$ where $\zeta$ admits a Taylor expansion
\begin{equation*}
\zeta(t+s,x+y)=\zeta(t,x)+\partial_t\zeta(t,x) s +\partial_x\zeta(t,x)y+\mathcal O(s^2+y^2),
\end{equation*}
the function $(s,y)\mapsto \zeta(t+s,x+y)-\varphi(s,y)$ with
\begin{equation*}
\varphi(s,y)= \partial_t\zeta(t,x) s +\partial_x\zeta(t,x)y + M (s^2 +y^2),
\end{equation*}
admits, provided $M$ is large enough, a local maximum at $(0,0)$. Since $\zeta$ is a viscosity subsolution, {by definition of the latter property (see e.g. \cite[Definition~2.2]{DOW04})} we deduce that at any such point $(t,x)$ it holds
\begin{equation*}
\partial_t\zeta(t,x) +\frac 12 (\partial_x\zeta(t,x))^2\leq -\delta.
\end{equation*}
Since $\zeta$ is semiconvex, Alexandrov's theorem {\cite{alexandrov39} (see also \cite[Theorem~3.11]{niculescupersson})} ensures that such an expansion is valid at almost every $(t,x)\in U$. In particular there exists a sequence $(t_k,x_k)\to 0$ such that $\zeta$ is differentiable at $(t_k,x_k)$ and
\begin{equation*}
\partial_t\zeta(t_k,x_k) +\frac 12 (\partial_x\zeta(t_k,x_k))^2\leq -\delta.
\end{equation*}
Up to extracting a subsequence we may assume that $(v_k,w_k):=\nabla \zeta(t_k,x_k)$ converges towards $(v,w)\in \R^2$ that satisfies
\begin{equation}\label{eq:vw}
v +\frac 12 w^2 \leq -\delta.
\end{equation} 
Moreover the convex function $\zeta+(2r)^{-1}((t-t_k)^2+(x-x_k)^2)$ lies above its tangent: for all $(t,x)\in U$ it holds
\begin{equation*}
\zeta(t,x)+\frac{1}{2r}((t-t_k)^2+(x-x_k)^2)\geq \zeta(t_k,x_k) + v_k (t-t_k) +w_k(x-x_k).
\end{equation*}
Passing to the limit yields \eqref{eq:subgradient} for
$\zeta_a(t,x)=\zeta(0,0) +v t +w x$,
and \eqref{eq:affinesubsol} follows from \eqref{eq:vw}.

\textbf{Step 2.} 
For any {height} $\eta>0$ with
{
\begin{equation}\label{eq:etarestrict1}
\eta \ll \delta^5 r,
\end{equation}
(where the symbol $\ll$ denotes inequality up to a small universal constant)},
letting 
\begin{equation*}
\widetilde \zeta(t,x) :=\zeta_a(t,x)-\frac{1}{2r}\abs{(t,x)}^2-\frac{\delta}{2r}\abs{(t,x)}^2,\qquad \abs{(t,x)}^2=t^2+x^2,
\end{equation*}
{ and defining as in \cite{DOW04} the set}
\begin{equation*}
\Omega_\eta:=B_\rho\cap \left\lbrace \widetilde \zeta +\eta \geq h\right\rbrace,
\end{equation*}
it holds
\begin{equation}\label{eq:boundsomegaeta}
B_{\eta/3}\subset\Omega_\eta\subset B_{2(r\eta/\delta)^{1/2}}{\subset\subset B_\rho}.
\end{equation}

Since $h$ and $\zeta_a$ are $1$-Lipschitz, in $B_{\eta/3}$ we obtain
\begin{align*}
\widetilde \zeta +\eta -h & \geq \eta-2\abs{(t,x)} -\frac{1+\delta}{2r}\abs{(t,x)}^2\\
&\geq \eta-\left(2+\frac{\eta}{3r} \right)\abs{(t,x)} \geq \eta-\left(2+\frac{\delta^5}{30}\right)\abs{(t,x)}\\
&\geq \eta-3\abs{(t,x)}>0,
\end{align*}
which implies $B_{\eta/3}\subset \Omega_\eta$.

{The strict inclusion $B_{2(r\eta/\delta)^{1/2}}\subset\subset B_\rho$ follows from \eqref{eq:etarestrict1} and $\rho\geq \delta^2 r$  which imply 
\begin{equation*}
2\left(\frac{r\eta}{\delta}\right)^{1/2}\ll \delta^2r <\rho.
\end{equation*}
Moreover} since $h\geq \zeta$ in $B_\rho$ and \eqref{eq:subgradient} holds, in $B_\rho\setminus B_{2(r\eta/\delta)^{1/2}}$ we have
\begin{equation*}
h-\widetilde \zeta \geq \frac{\delta}{2r}(t^2+x^2) \geq 2\eta,
\end{equation*}
which shows $\Omega_\eta\subset B_{2(r\eta/\delta)^{1/2}}$.

\textbf{Step 3.} Denoting by $\langle f \rangle$ the average of a function $f$ in $\Omega_\eta$ {and assuming
\begin{equation}\label{eq:etarestrict2}
\eta\ll \delta^3r,
\end{equation}
}
it holds
{
\begin{equation}\label{eq:bounddeltau4}
\delta\leq  \left\langle \left(u-\langle u\rangle\right)^2\right\rangle.
\end{equation} 
}

Using the definition of $\widetilde\zeta$ {in Step 2} together with \eqref{eq:affinesubsol}, \eqref{eq:etarestrict2} and the fact that {$\abs{\partial_x\zeta_a}\leq 1$}, in $B_{2(r\eta/\delta)^{1/2}}$  we find
\begin{align*}
\partial_t\widetilde\zeta +\frac 12 (\partial_x\widetilde\zeta)^2 & = \partial_t\zeta_a+\frac 12 (\partial_x\zeta_a)^2 -\frac{1+\delta}{r}t - 2\frac{1+\delta}{r}x\partial_x\zeta_a +\frac{(1+\delta)^2}{2r^2}x^2\\
& \leq -\delta +C \left(\frac{\eta}{\delta^3r}\right)^{1/2}\delta \leq -\frac\delta 2.
\end{align*}
By \eqref{eq:boundsomegaeta} this holds in particular in $\Omega_\eta$ and therefore using Jensen's inequality we have
\begin{equation*}
-\frac{\delta}{2} \geq \langle \partial_t\widetilde\zeta\rangle +\frac 12 \langle (\partial_x\widetilde\zeta)^2\rangle \geq 
\langle \partial_t\widetilde\zeta\rangle +\frac 12 \langle \partial_x\widetilde\zeta\rangle^2. 
\end{equation*}
{Moreover} since $u=\partial_x h$ {and $(\widetilde\zeta +\eta -h)_+$ has compact support in $B_\rho$} it holds
\begin{equation}\label{eq:meandxzetatilde}
\langle \partial_x\widetilde\zeta\rangle - \langle u\rangle  =
\frac{1}{\abs{\Omega_\eta}}\int_{B_\rho}\partial_x\left[ (\widetilde\zeta +\eta -h)_+\right] =0,
\end{equation}
and similarly $\langle \partial_t\widetilde\zeta\rangle=\langle -u^2/2\rangle$. This implies
\begin{equation*}
-\frac{\delta}{2}\geq -\frac 12 \left\langle \left(u-\langle u\rangle\right)^2\right\rangle ,
\end{equation*}
and proves \eqref{eq:bounddeltau4}.

\textbf{Step 4.} It holds
\begin{equation}\label{eq:boundu4mu}
{\left\langle \left(u-\langle u\rangle\right)^2\right\rangle^2}\lesssim \left(\frac{\eta}{\delta r}\right)^{1/2}+\frac{\norm{\mu_+}}{\eta},
\end{equation}
where the symbol $\lesssim$ stands for inequality up to a universal constant.

{The argument relies as in \cite{DOW04} on a quantification of Tartar's application of the div-curl lemma to equations of Burgers type \cite{tartar83}.
{By H\"older's inequality and} \cite[Proposition~3.2]{DOW04} we have
\begin{align*}
\left\langle \left(u-\langle u\rangle\right)^2\right\rangle^2 & 
\leq \left\langle \left(u-\langle u\rangle\right)^4\right\rangle\\
& \lesssim \left\langle
\left(
\begin{array}{c}
-\frac{u^2}{2} \\
u
\end{array}
\right)
\cdot
\left(
\begin{array}{c}
u^2/2 \\
u^3/3
\end{array}
\right)
 \right\rangle -\left\langle
 \left(
\begin{array}{c}
-\frac{u^2}{2} \\
u
\end{array}
\right) \right\rangle\cdot\left\langle
\left(
\begin{array}{c}
u^2/2 \\
u^3/3
\end{array}
\right)
 \right\rangle \\
& = \left\langle
\left(
\begin{array}{c}
\partial_t h \\
\partial_x h
\end{array}
\right)
\cdot
\left(
\begin{array}{c}
u^2/2 \\
u^3/3
\end{array}
\right)
 \right\rangle -\left\langle
 \left(
\begin{array}{c}
\partial_t h \\
\partial_x h
\end{array}
\right) \right\rangle\cdot\left\langle
\left(
\begin{array}{c}
u^2/2 \\
u^3/3
\end{array}
\right)
 \right\rangle 
\end{align*}
Recalling \eqref{eq:meandxzetatilde} and its counterpart for the $t$-derivative, we deduce
\begin{align*}
\left\langle \left(u-\langle u\rangle\right)^2\right\rangle^2 & 
\lesssim\left\langle
 \left(
 \begin{array}{c}
\partial_t (h-\widetilde\zeta-\eta) \\
\partial_x (h-\widetilde\zeta-\eta) 
\end{array}
 \right)
 \cdot
\left(
\begin{array}{c}
u^2/2 \\
u^3/3
\end{array}
\right)
 \right\rangle \\
 &\quad +
 \left\langle
\left(
\begin{array}{c}
\partial_t \widetilde\zeta-\langle\partial_t\widetilde\zeta\rangle \\
\partial_x \widetilde\zeta-\langle\partial_x\widetilde\zeta\rangle
\end{array}
\right)
\cdot
\left(
\begin{array}{c}
u^2/2 \\
u^3/3
\end{array}
\right)
 \right\rangle\\
 &\leq \frac{1}{\abs{\Omega_\eta}}\int_{\Omega_\eta}(\widetilde\zeta+\eta-h)d\mu + \frac{2}{r}\diam(\Omega_\eta)
\end{align*} 
For the last inequality we used the fact that, $\zeta_a$ being affine, we have 
\begin{equation}
\nabla \widetilde \zeta -\langle\nabla\widetilde\zeta\rangle = -\frac{1+\delta}{r}\left(\begin{array}{c}t \\ x\end{array}\right). 
\end{equation}
Since in $\Omega_\eta$ it holds $0\leq \widetilde\zeta+\eta-h\leq \eta+\zeta-h\leq \eta$, we find
\begin{equation*}
\left\langle \left(u-\langle u\rangle\right)^2\right\rangle^2  
\lesssim \frac{\eta}{\abs{\Omega_\eta}}\norm{\mu_+}+\frac 2r \diam(\Omega_\eta).
\end{equation*}
Using the inclusions \eqref{eq:boundsomegaeta} satisfied by $\Omega_\eta$ we obtain \eqref{eq:boundu4mu}.}

\textbf{Step 5.} Conclusion. 

We choose $\eta={(\delta r)^{1/3}}\norm{\mu_+}^{2/3}$ in order to balance the two terms on the right-hand side of \eqref{eq:boundu4mu}. Since
\begin{equation*}
{(\delta r)^{1/3}}\norm{\mu_+}^{2/3}\leq \frac{1}{ C^{2/3}} \delta^5 r,
\end{equation*}
the restrictions \eqref{eq:etarestrict1} and \eqref{eq:etarestrict2} on $\eta$ are indeed satisfied provided $C$ is large enough.
Moreover combining \eqref{eq:bounddeltau4} and \eqref{eq:boundu4mu} with the smallness assumption \eqref{eq:mu+small} on $\mu_+$, we obtain
\begin{equation*}
\delta^2 \lesssim {\left(\frac{\norm{\mu_+}}{\delta r}\right)^{1/3}}\lesssim\frac{1}{ C^{1/3}} \delta^2,
\end{equation*}
and therefore the desired contradiction for large enough $C$.
\end{proof}

\begin{proof}[Proof of Proposition~\ref{prop:errorvisc}.]
Note that since $h\leq\bar h$  we only need to estimate $(h-\bar h)$ from below.
Given $\rho\in (0,1)$ we consider the sup-convolution
\begin{equation*}
\bar h_\rho(t,x)=\sup_{(s,y)\in Q}\left\lbrace \bar h(s,y)-\frac{1}{2\rho}\left((t-s)^2+(x-y)^2\right)\right\rbrace.
\end{equation*}
{As a supremum of $(1/\rho)$-semiconvex functions }this function $\bar h_\rho$ is $(1/\rho)$-semiconvex.
We also introduce parameters $\delta\in (0,1)$, $H,M\geq 1$ and define
\begin{equation*}
\zeta(t,x)=\bar h_\rho(t,x) - \delta t - \frac{M}{2} ((t-t_1)_+)^2 - H\rho,
\end{equation*}
so that $\zeta$ is $(1/r)$-semiconvex with $1/r=1/\rho + M$. We want to use Lemma~\ref{lem:supersol} to deduce that $h\geq\zeta$ in $Q$ and from there obtain the desired lower bound on $(h-\bar h)$. {We split the proof in the following way : in Step 1 we prove that $\zeta$ is a viscosity subsolution as in Lemma~\ref{lem:supersol}; then we show that $h\geq \zeta$ near the boundary, dealing with the parabolic boundary $\partial_0 Q$ in Step 2 and the remaining boundary in Step 3; in Step 4 we check that the Lipschitz constant of $\zeta$ depends only on $L$ and $t_1$ in the relevant region; eventually in Step 5 we apply Lemma~\ref{lem:supersol} and optimize the choices of $\rho$ and $\delta$ in order to conclude.}

\textbf{Step 1.} For {$\rho\ll 1/L$}, the function $\zeta$ is a viscosity subsolution of
\begin{equation}\label{eq:zetasubsol}
\partial_t\zeta +\frac 12 (\partial_x\zeta)^2\leq -\delta\quad\text{ in }\widetilde Q :=Q\cap \lbrace  \dist(\cdot,\partial Q)> 4L\rho\rbrace.
\end{equation}

It suffices to show that $\bar h_\rho$ is a viscosity subsolution of
\begin{equation}\label{eq:barhrhosubsol}
\partial_t\bar h_\rho + \frac 12 (\partial_x\bar h_\rho)^2\leq 0\quad\text{ in }\widetilde Q.
\end{equation}
{The fact that sup convolution preserves the viscosity subsolution property is well-known, see e.g. \cite[Lemma~A.5]{crandallishiilions92}. For the convenience of the reader we provide a proof of \eqref{eq:barhrhosubsol} in our setting.}
Let $\varphi(t,x)$ be a smooth function such that $(\bar h_\rho -\varphi)$ attains its maximum at $(t_0,x_0)\in\widetilde Q$, and assume w.l.o.g. that $\varphi(t_0,x_0)=\bar h_\rho(t_0,x_0)$. For any $(s,y)\in U$ with 
\begin{equation*}
d:=\abs{(s,y)-(t_0,x_0)}=\sqrt{(t_0-s)^2+(x_0-y)^2 } \geq 2L\rho,
\end{equation*}
since by \eqref{eq:lipbarh} we have $\Lip(\bar h)\leq L$, it holds
\begin{align*}
\bar h(s,y)-\frac{1}{2\rho}d^2 &\leq \bar h(t_0,x_0)+
\left(L-\frac{d}{2\rho}\right) d \\
& \leq \bar h(t_0,x_0)\leq \bar h_\rho(t_0,x_0).
\end{align*}
Hence the supremum in the definition of $\bar h_\rho(t_0,x_0)$ is attained at some $(s_0,y_0)\in B_{2L\rho}(t_0,x_0)\subset Q$, and
\begin{equation}\label{eq:phit0x0}
\varphi(t_0,x_0)={\bar h_\rho(t_0,x_0)}=\bar h(s_0,y_0)-\frac{1}{2\rho}\left((t_0-s_0)^2+(x_0-y_0)^2\right).
\end{equation}
Moreover since $(\bar h_\rho -\varphi)$ is maximal at $(t_0,x_0)$ with value zero, it holds
\begin{equation}\label{eq:phigeqbarhrho}
\bar h(s,y)-\frac{1}{2\rho}\left((t-s)^2 +(x-y)^2\right) \leq \varphi(t,x)\qquad\forall (t,x),(s,y)\in Q.
\end{equation}
In particular for all $(s,y)\in B_{2L\rho}(s_0,y_0)\subset Q$ we may choose 
\begin{equation*}
(t,x)=(s-s_0+t_0,y-y_0 +x_0)\in B_{2L\rho}(t_0,x_0)\subset Q,
\end{equation*}
in \eqref{eq:phigeqbarhrho} and obtain
\begin{equation*}
\bar h(s,y)\leq  \varphi(s-s_0+t_0,y-y_0 +x_0)
+\frac{1}{2\rho}\left((t_0-s_0)^2 +(x_0-y_0)^2\right)=:\psi(s,y).
\end{equation*}
Moreover \eqref{eq:phit0x0} ensures $\psi(s_0,y_0)=\bar h(s_0,y_0)$, hence $\bar h-\psi$ has a local maximum at $(s_0,y_0)$. {Since $\bar h$ is} a viscosity solution we deduce that
\begin{equation*}
\partial_t\varphi(t_0,x_0)+\frac 12 (\partial_x\varphi(t_0,x_0))^2 = \partial_t\psi(s_0,y_0)+\frac 12 (\partial_x\psi(s_0,y_0))^2\leq 0,
\end{equation*} 
which proves \eqref{eq:barhrhosubsol}.

\textbf{Step 2.} Provided $H$ is large enough {(depending only on $L$)} it holds
\begin{equation*}
\zeta\leq h \quad\text{ in }Q\cap \left\lbrace \dist(\cdot,\partial_0 Q)\leq 4L\rho + \delta^2 r\right\rbrace.
\end{equation*}

The definition of $\bar h_\rho$ implies
\begin{equation*}
\bar h_\rho(t,x)\leq \bar h(t,x)+\frac{L^2}{2}\rho.
\end{equation*}
By definition of $\zeta$ this yields
\begin{equation*}
\zeta(t,x)-h(t,x)\leq -\frac M2 ((t-t_1)_+)^2 -H\rho +\frac{L^2}{2}\rho+\bar h(t,x)-h(t,x),
\end{equation*}
so that by $\bar h=h$ on $\partial_0 Q$ and the Lipschitz continuities of $h$ and $\bar h$ \eqref{eq:lipbarh} we obtain
\begin{equation}\label{eq:diffzetah}
\zeta(t,x)-h(t,x)\leq -\frac M2 ((t-t_1)_+)^2 -H\rho +\frac{L^2}{2}\rho+ 2L\dist((t,x),\partial_0 Q).
\end{equation}
Therefore if $\dist((t,x),\partial_0 Q)\leq 4L\rho+ \delta^2r\leq 5L\rho$ we have
\begin{align*}
\zeta(t,x)-h(t,x) & \leq \bar h_\rho(t,x)-\bar h(t,x) +\bar h(t,x)-h(t,x) -H\rho\\
&\leq \left({\frac{L^2}{2}}+10L^2-H\right)\rho,
\end{align*}
and it suffices to choose $H\geq  11L^2$.

\textbf{Step 3.} Provided $M$ is large enough {(depending only on $L$ and $t_1$)} it holds 
\begin{equation*}
\zeta\leq h\quad\text{ in }Q\cap{\lbrace t\geq (1+t_1)/2\rbrace}.
\end{equation*}

{For $t\geq (1+t_1)/2$ we have by \eqref{eq:diffzetah}
\begin{align*}
\zeta(t,x)-h(t,x)&\leq \frac{L^2}{2}\rho+2L-\frac M8 (1-t_1)^2\\
&\leq \frac{L^2}{2}+2L-\frac{M}{8}(1-t_1)^2,
\end{align*}
and it suffices to choose $M\geq (4L^2+16L)/{(1-t_1)^2}$.}

{
\textbf{Step 4.} We have
\begin{equation*}
\Lip(\zeta)\leq 2L + 2M \qquad\text{in }\widetilde Q,
\end{equation*}
where $\widetilde Q = Q\cap \lbrace \dist(\cdot,\partial Q)>4L\rho\rbrace$ as in Step 1.
}

{
It was shown in Step 1 that for $(t_0,x_0)$ in $\widetilde Q$, the supremum in the definition of $\bar h_\rho(t_0,x_0)$ is attained at some $(s_0,y_0)\in B_{2L\rho}(t_0,x_0)$. It follows that for any small $(t,x)$ we have
\begin{align*}
\bar h_\rho(t_0,x_0)& -\bar h_\rho(t_0+t,x_0+x) \\
&=\bar h(s_0,y_0)-\frac{1}{2\rho}\abs{(t_0-s_0,x_0-y_0)}^2 \\
&\quad -\sup_{(s,y)\in Q}\left\lbrace \bar h(s,y)-\frac{1}{2\rho}\abs{(t_0+t-s,x_0+x-y)}^2   \right\rbrace \\
&\leq \frac{1}{2\rho}\left(2 (t_0-s_0)t + 2(x_0-y_0)x + \abs{(t,x)}^2\right)\\
&\leq {2L}\abs{(t,x)} + \frac{1}{2\rho}\abs{(t,x)}^2.
\end{align*}
This implies $\abs{\nabla\bar h_\rho}\leq 2L$ in $\widetilde Q$. Therefore in $\widetilde Q$ it holds
\begin{align*}
\Lip(\zeta) & \leq 2L + \delta + M,
\end{align*}
which concludes the proof of Step 4 since $\delta\leq 1$ and $M\geq 1$. }

\textbf{Step 5.} Conclusion.

{Recalling \eqref{eq:zetasubsol} and Step 4, Lemma~\ref{lem:supersol} ensures the existence of a constant $C>0$ depending through $\Lip(h)$ and $\Lip(\zeta)$ only on $L$ and $t_1$, such that  if $C\norm{\mu_+}\leq \delta^7 r$ then the minimum of $(h-\zeta)$ in $\overline Q$ cannot be attained at any $x\in \widetilde Q$ such that $B_{\delta^2r}(x)\subset\widetilde Q$. Moreover if $4L\rho +\delta^2r\leq (1-t_1)/2$ then by Steps 2 and 3, at any $x\in\widetilde Q$ such that $B_{\delta^2r}(x)$ is not contained in $\widetilde Q$ it must hold $h-\zeta\geq 0$.} 

{
Since $\delta^2 r\leq \rho$ we deduce that if $\rho\leq\rho_0:=(1-t_1)/(8L+2)$ and $C\norm{\mu_+}\leq \delta^7 r$ then $h-\zeta\geq 0$ in $\overline Q$.}
Hence for $t\leq t_1$ it holds
\begin{align*}
h-\bar h 
& \geq \zeta-\bar h \\
& =  \bar h_\rho -\bar h -\delta t - H\rho \\
&\geq - H \rho -\delta,
\end{align*}
and therefore
\begin{equation*}
\sup_{Q\cap\lbrace t\leq t_1\rbrace}\abs{h-\bar h}\leq H \rho +\delta.
\end{equation*}
Choosing $\delta=\rho$ and recalling that $r=\rho/(1+\rho M)\geq \rho/(1+M)\geq\rho/(2M)$ we conclude that 
\begin{equation*}
2MC\norm{\mu_+}\leq {\rho^8} \quad\Longrightarrow \quad \sup_{Q\cap\lbrace t\leq t_1\rbrace}\abs{h-\bar h}\leq (1+H)\rho.
\end{equation*}
If $\norm{\mu_+}\leq {\rho_0^8}/(2MC)$ we can apply this to {$\rho=(2MC\norm{\mu_+})^{1/8}$} to finish the proof. If $\norm{\mu_+}>{\rho_0^8}/(2MC)$ we can simply invoke the fact that $\abs{h- \bar h}$ is bounded by a constant depending only on $L$.
\end{proof}

\section{Proofs of Theorems~\ref{thm:errorentropy} and \ref{thm:lebquantit}}\label{s:proofs}

In this section we use the symbol $\lesssim$ to denote inequality up to a constant depending only on $\norm{u}_\infty$. 

\begin{proof}[Proof of Theorem~\ref{thm:errorentropy}.]
Let $h$ be the Lipschitz solution of \eqref{eq:HJ} such that $u=\partial_x h$ {and $-u^2/2=\partial_t h$}, let $\bar h$ be as in \S\ref{s:HJ}  the viscosity solution of \eqref{eq:HJ} in $Q_1$ with $\bar h =h$ on the parabolic boundary $\partial_0 Q_1$. By Proposition~\ref{prop:errorvisc} it holds
\begin{equation}\label{eq:weakestim}
\sup_{Q_{7/8}}\abs{\bar h -h}\lesssim \mu_+(Q_1)^{\frac 18}.
\end{equation}
This estimate tells us that $u$ is close to the entropy solution $\zeta=\partial_x \bar h$ in a weak sense. The rest of the proof consists in transforming this weak estimate into the desired  $L^1$-estimate: in a first step we quantify the compactness of solutions of \eqref{eq:burgers} satisfying \eqref{eq:mu}, and in a second step we use this quantitative compactness and a standard interpolation argument to conclude.

\textbf{Step 1.} We show that
\begin{equation}\label{eq:compactquant}
\sup_{\abs{\xi}\leq r/4}\int_{Q_{3/4}}(u(\cdot +\xi) -u)^4 \lesssim r(1+\abs{\mu}(Q_1))+\frac 1r  \mu_+(Q_1)^{\frac 18},
\end{equation}
for all $r\in (0,1/8)$.

The proof relies on a \enquote{div-curl} argument used also in \cite{giacomelliotto05}. First we use the Galilean invariance of \eqref{eq:burgers}: for any  constant $c\in\R$ the function $\tilde u =u-c$ satisfies
\begin{align*}
(\partial_t +c\partial_x)\tilde u +\partial_x \frac {\tilde u^2}{2} &=0,\\
(\partial_t+c\partial_x) \frac {\tilde u^2}{2} +\partial_x \frac{\tilde u^3}{3} & =\mu,
\end{align*}
and $\tilde u=\partial_x\tilde h$, $-\tilde u^2/2=(\partial_t + c\partial_x) \tilde h$, where
\begin{equation*}
\tilde h = h-cx+\frac 12 c^2t.
\end{equation*}
We infer that for any $z_0\in Q_{3/4}$ it holds
\begin{align*}
\frac 1{12} \tilde u^4 & =\left(
\begin{array}{c}
\frac 12 \tilde u^2\\
\frac 13 \tilde u^3
\end{array}
\right)
\cdot
\left(
\begin{array}{c}
(\partial_t+c\partial_x)(\tilde h-\tilde h(z_0))\\
\partial_x(\tilde h-\tilde h(z_0))
\end{array}
\right) \\
& = \left(
\begin{array}{c}
\partial_t+c\partial_x\\
\partial_x
\end{array}
\right)\cdot\left[
(\tilde h-\tilde h(z_0))\left(
\begin{array}{c}
\frac 12 (u-c)^2\\
\frac 13 (u-c)^3
\end{array}
\right)\right]\\
&\quad
-(\tilde h-\tilde h(z_0))\mu.
\end{align*}
We multiply this identity by a smooth cut-off function at scale $r\in (0,1/8)$ and deduce, restricting
 ourselves to $\abs{c}\lesssim 1$ and recalling that $\Lip(\bar h)\lesssim 1$, 
\begin{equation}\label{eq:tildeu4}
\int_{Q_{r/2}(z_0)}\tilde u^4 \lesssim \frac 1r \int_{Q_{r}(z_0)}\abs{\tilde h-\tilde h(z_0)}+r\abs{\mu}(Q_{r}(z_0)).
\end{equation}
To estimate the first term on the right-hand side we wish to pass from $h$ to the viscosity solution $\bar h$,  since it is semi-concave and in particular twice differentiable almost everywhere. In fact in $Q_{7/8}$ the second derivative of $\bar h$ in any direction is bounded from above by a universal constant \cite[Theorem~13.1]{lions} and therefore, in conjunction with $\Lip(\bar h)\lesssim 1$,
\begin{equation}\label{eq:d2barh}
\int_{Q_{7/8}}\abs{\nabla^2 \bar h}\lesssim 1.
\end{equation}
We write
\begin{align*}
\int_{Q_{r}(z_0)}\abs{\tilde h-\tilde h(z_0)}& \lesssim r^2\sup_{Q_{7/8}}\abs{h-\bar h} \\
&\quad + \int_{Q_{r}(z_0)}\abs{\bar h-\left(\bar h(z_0)+c(x-x_0)-\frac 12 c^2 (t-t_0)\right)}.
\end{align*}
Next we assume that $\bar h$ is differentiable at $z_0$ (which is the case for almost every $z_0$) and choose $c=c(z_0)=\partial_x\bar h(z_0)$, so that $\partial_t \bar h (z_0)=-c^2/2$ and, using also \eqref{eq:weakestim} the above turns into
\begin{align*}
\int_{Q_{r}(z_0)}\abs{\tilde h-\tilde h(z_0)}& \lesssim r^2\sup_{Q_{7/8}}\abs{h-\bar h} \\
&\quad
+ \int_{Q_r(z_0)}\abs{\bar h(z) -(h(z_0)+\nabla \bar h(z_0)\cdot (z-z_0))}\, dz\\
&\lesssim r^2 \left[ \mu_+(Q_1)\right]^{1/8} + r^2 \int_0^1 \int_{Q_r}\abs{\nabla^2\bar h(z_0+s\tilde z)}\, d\tilde z ds.
\end{align*}
Plugging this back into \eqref{eq:tildeu4} we deduce that for $\abs{\xi}\leq r/4$ it holds
\begin{align*}
\int_{Q_{r/4}(z_0)}(u(\cdot+\xi)-u)^4 &\lesssim \int_{Q_{r/2}(z_0)}(u-c(z_0))^4\\
& \lesssim r\abs{\mu}(Q_{r}(z_0)) + r \left[ \mu_+(Q_1)\right]^{1/8}\\
&\quad
+ r \int_0^1 \int_{Q_r}\abs{\nabla^2\bar h(z_0+s\tilde z)}\, d\tilde z ds.
\end{align*}
Integrating over $z_0\in Q_{3/4}$, dividing by $r^2$, and recalling \eqref{eq:d2barh} we obtain \eqref{eq:compactquant}.

\textbf{Step 2.} Conclusion.

Recall that $\zeta$ is the entropy solution of \eqref{eq:burgers} given by $\zeta=\partial_x\bar h$. Next we use the compactness \eqref{eq:compactquant} to turn \eqref{eq:weakestim} into an $L^1$-estimate on
\begin{equation*}
u-\zeta=\partial_x(h-\bar h).
\end{equation*}
We introduce a smooth kernel $\varphi(z)$ with compact support in $Q_1$ and unit integral and define
\begin{equation*}
u_r=u\star \varphi_r,\qquad \zeta_r =\zeta\star\varphi_r,\qquad\text{where } \varphi_r(z)=r^{-2}\varphi(z/r).
\end{equation*}
From \eqref{eq:compactquant} we deduce that for $r\in (0,1/8)$ it holds
\begin{equation*}
\int_{Q_{3/4}}\abs{u-u_r}^4\lesssim \e(1+\abs{\mu}(Q_1))+\frac 1r \mu_+(Q_1)^{\frac 18}.
\end{equation*}
Moreover, \eqref{eq:d2barh} implies that
\begin{equation*}
\int_{Q_{3/4}}\abs{\zeta-\zeta_r}^4\lesssim r^4.
\end{equation*}
The combination yields
\begin{align*}
\int_{Q_{3/4}}\abs{u-\zeta}^4&
\lesssim r^4 + r(1+\abs{\mu}(Q_1))+\frac 1r  \mu_+(Q_1)^{\frac 18} 
+\int_{Q_{3/4}}\abs{(h-\bar h)\star\partial_x\varphi_r}^4\\
&\lesssim r^4 + r(1+\abs{\mu}(Q_1))+\frac 1r  \mu_+(Q_1)^{\frac 18}
+\frac {1}{r^4} \mu_+(Q_1)^{\frac 12},
\end{align*}
where we used \eqref{eq:weakestim} in the second step.
We choose $r=[\mu_+(Q_1)]^{1/16}(1+\abs{\mu}(Q_1))^{-1/5}$, which is admissible since without loss of generality $\mu+(Q_1)\ll 1$, to find our conclusion
\begin{equation*}
\int_{Q_{3/4}}\abs{u-\zeta}^4 \lesssim (1+\abs{\mu}(Q_1))^{4/5}\left[\mu_+(Q_1)\right]^{1/16}.
\end{equation*}
\end{proof}

\begin{proof}[Proof of Theorem~\ref{thm:lebquantit}.]
\textbf{Step 1.} If $u$ is as in Theorem~\ref{thm:errorentropy}, then for any $\theta\in (0,3/4)$  it holds
\begin{equation}\label{eq:pfth2st1}
\fint_{\abs{t}\leq \theta}\fint_{\abs{x}\leq \theta}\fint_{\abs{y}\leq\theta} \abs{u(t,x)-u(t,y)}\,dx dy dt\lesssim \theta +\frac{1}{\theta^3}\min(1,\abs{\mu}(Q_1)^{\frac{1}{64}}).
\end{equation}
To prove this estimate we apply Theorem~\ref{thm:errorentropy} to $u$ and to $-u(-t,x)$ and deduce the existence of an entropy solution $\overline\zeta$ and an anti-entropy solution $\underline\zeta$ of \eqref{eq:burgers} in $Q_1$ with
\begin{equation*}
\int_{Q_{3/4}}\abs{u-\overline\zeta}+\int_{Q_{3/4}}\abs{u-\underline\zeta}\lesssim (1+\abs{\mu}(Q_1))^{\frac 15}\abs{\mu}(Q_1)^{\frac{1}{64}}.
\end{equation*}
Since $\overline\zeta$ is an entropy solution and $\underline\zeta$ an anti-entropy solution it holds by Oleinik's principle
\begin{equation*}
\partial_x\overline\zeta\lesssim 1\quad\text{and}\quad -\partial_x\underline\zeta\lesssim 1\qquad\text{ in }Q_{3/4}.
\end{equation*}
Combining these facts we compute
\begin{align*}
\fint_{\abs{t}\leq \theta}\fint_{\abs{x}\leq \theta}\fint_{\abs{y}\leq\theta} \abs{u(t,x)-u(t,y)} & \lesssim \frac{1}{\theta^3}\int_{Q_{3/4}}\abs{u-\overline\zeta}+\frac{1}{\theta^3}\int_{Q_{3/4}}\abs{u-\underline\zeta} \\
&\quad + \fint_{\abs{t}\leq \theta}\frac{1}{\theta^2}\iint_{\theta>x>y>-\theta}(\overline\zeta(t,x)-\overline\zeta(t,y))_+  \\
&\quad + \fint_{\abs{t}\leq \theta}\frac{1}{\theta^2}\iint_{-\theta<x<y<\theta}(\underline\zeta(t,x)-\underline\zeta(t,y))_+\\
&\lesssim \frac{1}{\theta^3}(1+\abs{\mu}(Q_1))^{\frac 15}\abs{\mu}(Q_1)^{\frac{1}{64}} +\theta.
\end{align*}
This proves \eqref{eq:pfth2st1} if $\abs{\mu}(Q_1)\leq 1$, and for $\abs{\mu}(Q_1)\geq 1$ the left-hand side of \eqref{eq:pfth2st1} is $\lesssim 1$ so that \eqref{eq:pfth2st1} holds.

\textbf{Step 2.} If $u$ is as in Theorem~\ref{thm:lebquantit} and $z_0=(t_0,x_0)\in\Omega_{\alpha,K}$ then it holds
\begin{equation*}
D(r):=\fint_{\abs{t-t_0}\leq r}\fint_{\abs{x-x_0}\leq r}\fint_{\abs{y-x_0}\leq r} \abs{u(t,x)-u(t,y)}\,dx dy dt\lesssim C(K,d(z_0),\alpha) r^{\frac{\alpha}{\alpha+256}}
\end{equation*}
for all $r\in (0,d(z_0))$.

We assume without loss of generality $K\geq 1$.
For any $\rho\in (0,d(z_0)\wedge K^{-1/\alpha})$ we apply Step 1 to $z\mapsto u(\rho(z-z_0))$ and obtain
\begin{equation*}
D\left(\frac{3}{4}\theta\rho\right)\lesssim\theta +\frac{1}{\theta^3}K^{\frac{1}{64}}\rho^{\frac{\alpha}{64}}\qquad\text{for all }\theta\in (0,1).
\end{equation*}
We  choose $\theta=\left(K^{\frac {1}{64}}\rho^{\frac{\alpha}{64}}\right)^{\frac 14}$ to balance the two terms and set 
$r=\frac 34 \theta\rho=\frac 34 K^{\frac 1{256}}\rho^{\frac{256+\alpha}{256}}$. This yields
\begin{equation*}
D(r)\lesssim K^{\frac{1}{\alpha+256}}
r^{\frac{\alpha}{\alpha+256}}.
\end{equation*}
This is valid for $r\leq r_0(K,d(z_0))$ due to the constraints on $\rho$, and for $r\geq r_0$ we have $D(r)\lesssim (r/r_0)^{\frac{\alpha}{\alpha+256}}$.

\textbf{Step 3.} 
From Step 2 we have the desired regularity in the space variable, and it remains to use the equation \eqref{eq:burgers} to transfer it to the time variable. We sketch here the standard argument.

 We use coordinates in which $z_0=0$, we fix a smooth cut-off function $\eta(x)$, set $\eta_r(x)=r^{-1}\eta(r^{-1}x)$ and notice that
\begin{equation*}
\frac{d}{dt}\left[ \int u(t,x)\eta_r(x) dx\right]=
\frac 12 \int u^2(t,x)(\eta_r)'(x) dx.
\end{equation*}
In particular $t\mapsto \int u(t,x)\eta_r(x) dx$ is Lipschitz, and it holds
\begin{align*}
&\int (u(s,x)-u(t,x))\eta_r(x) dx  = \frac {s-t}2\int_0^1 \int u^2(\tau s +(1-\tau)t,x)(\eta_r)'(x)dx\, d\tau\\
&\quad =\frac{s-t}{2r} \int_0^1 \int\left(  u^2(\tau s +(1-\tau)t,x)-u^2(\tau s +(1-\tau)t,y)\right)(\eta')_r(x) dx \, d\tau,
\end{align*}
where we used the fact that $\eta'$ has zero average and where $y$ can be choosen arbitrarily. Hence together with $\abs{u}\lesssim 1$ we obtain from averaging over $\abs{t}\leq r$, $\abs{s}\leq r$ and $\abs{y}\leq r$,
\begin{equation*}
\fint_{\abs{s}\leq r}\fint_{\abs{t}\leq r}\abs{\int (u(s,x)-u(t,x))\eta_r(x)\, dx } \lesssim D(r).
\end{equation*}
The conclusion then follows from Step 2.
\end{proof}

\appendix

\section{Lipschitz estimate for the viscosity solution}\label{a:lipvisc}

Let $h\in W^{1,\infty}(Q)$ solve \eqref{eq:HJ} almost everywhere.
The viscosity solution $\bar h\in W^{1,\infty}(Q)$ of \eqref{eq:HJ} with $\bar h=h$ on $\partial_0 Q$ is given \cite[\S{11}]{lions} by the Hopf-Lax formula 
\begin{equation*}
\bar h(t,x)=\inf\left\lbrace 
h(s,y)+\frac{(x-y)^2}{2(t-s)}\colon (s,y)\in\partial_0 Q,\; s<t
\right\rbrace.
\end{equation*}
Note that for $(t,x)\in Q$ the infimum is attained.
Let
$L:=\norm{\partial_x h}_{L^\infty(Q)}$,
so that the initial data $h(0,\cdot)$ has Lipschitz constant $\leq L$ and the boundary data $h(\cdot,0)$ and $h(\cdot,1)$ have Lipschitz constants $\leq L^2/2$.

\begin{lem}\label{lem:lipvisc}
It holds $\abs{\partial_x\bar h}\leq L$ a.e.
\end{lem}

\begin{proof}
Let $(t_0,x_0)\in Q$ and denote by $(s_0,y_0)$ a point at which the infimum defining $\bar h(t_0,x_0)$ is attained. Then for any small $x$ it holds
\begin{align*}
\bar h(t_0,x_0+x)-\bar h(t_0,x_0) & =\bar h(t_0,x_0+x)-h(s_0,y_0)-\frac{(x_0-y_0)^2}{2(t_0-s_0)}\\
&\leq \frac{(x_0+x-y_0)^2}{2(t_0-s_0)} -\frac{(x_0-y_0)^2}{2(t_0-s_0)}\\
& = \frac{x_0-y_0}{t_0-s_0} x + \frac{1}{2(t_0-s_0)}x^2,
\end{align*}
so that $\abs{\partial_ x\bar h(t_0,x_0)}\leq \abs{x_0-y_0}/(t_0-s_0)$ and to prove \eqref{eq:lipbarh} it suffices to show that the infimum defining  $\bar h(t_0,x_0)$ is attained at some $(s_0,y_0)$ with
\begin{equation}\label{eq:s0y0cone}
\frac{\abs{x_0-y_0}}{t_0-s_0}\leq L.
\end{equation}
We show that for any $(s,y)\in\partial_0 Q\cap\lbrace s<t_0\rbrace$ with 
\begin{equation}\label{eq:failsycone}
\frac{\abs{x_0-y}}{t_0-s} > L,
\end{equation}
there exists $(\tilde s ,\tilde y)\in\partial_0 Q\cap\lbrace s<t_0\rbrace$ satisfying
\begin{align}
&\frac{\abs{x_0-\tilde y}}{t_0-\tilde s} < \frac{\abs{x_0-y}}{t_0-s}\label{eq:tildesycone}\\
\text{and}\quad &h(s_0,\tilde y)+\frac{(x_0-\tilde y)^2}{2(t_0-\tilde s)} \leq  h(s_0, y)+\frac{(x_0-y)^2}{2(t_0-s)},\label{eq:tildesyinf}
\end{align}
which proves \eqref{eq:s0y0cone}. 

There are two cases to consider, depending on which part of the parabolic boundary $(s,y)$ belongs to.

\textbf{Case 1 :} $(s,y)\in \lbrace 0 \rbrace\times [0,1]$. We look for $(\tilde s,\tilde y)$ defined through $\tilde s=0$ and
\begin{equation*}
\frac{x_0-\tilde y}{t_0}=(1-\e)D,\quad D:=\frac{x_0-y}{t_0}, 
\end{equation*}
for some small $\epsilon>0$,
so that \eqref{eq:tildesycone} is satisfied. On the other hand since $h(0,\cdot)$ has Lipschitz constant $\leq L$, to show \eqref{eq:tildesyinf} it suffices to establish
\begin{equation*}
L\abs{\tilde y-y}  \leq \frac{(x_0-y)^2-(x_0-\tilde y)^2}{2t_0}
\quad\Longleftrightarrow\quad \frac L D \leq (1- \frac 12 \e),
\end{equation*}
which is satisfied for small enough $\e$ since \eqref{eq:failsycone} amounts to $\abs{D}>L$.

\textbf{Case 2 :} $(s,y)\in (0,1)\times \lbrace 0,1 \rbrace$. We assume $y=0$, the case $y=1$ being similar. We look for $(\tilde s,\tilde y)$ defined through $\tilde y=0$ and
\begin{equation*}
\frac{x_0}{t_0-\tilde s}=(1-\e) D,\quad D:=\frac{x_0}{t_0-s},
\end{equation*}
for some small $\epsilon>0$, so that \eqref{eq:tildesycone} is satisfied. On the other hand since $h(1,\cdot)$ has Lipschitz constant $\leq L^2/2$, to show \eqref{eq:tildesyinf} it suffices to establish
\begin{equation*}
\frac{L^2}{2}\abs{s-\tilde s}\leq \frac{x_0^2}{2}\left(\frac{1}{t_0-s}-\frac{1}{t_0-\tilde s}\right)\quad\Longleftrightarrow\quad \frac{L^2}{D^2}\leq  1-\e,
\end{equation*}
which is satisfied for small enough $\e$ since \eqref{eq:failsycone} amounts to $\abs{D}>L$.
\end{proof}

\bibliographystyle{plain}
\bibliography{ref}

\end{document}